\documentclass{amsart}

\usepackage[latin1]{inputenc}
\usepackage{amsfonts}
\usepackage{amsmath}
\usepackage{amsthm}
\usepackage{amssymb}
\usepackage{latexsym}
\usepackage{enumerate}

\newtheorem{theorem}{Theorem}
\newtheorem{lemma}[theorem]{Lemma}
\newtheorem{proposition}[theorem]{Proposition}

\newtheorem{definition}[theorem]{Definition}

\numberwithin{equation}{section}

\begin{document}

\newcommand{\cc}{\mathfrak{c}}
\newcommand{\N}{\mathbb{N}}
\newcommand{\C}{\mathbb{C}}
\newcommand{\Q}{\mathbb{Q}}
\newcommand{\R}{\mathbb{R}}
\newcommand{\T}{\mathbb{T}}
\newcommand{\st}{*}
\newcommand{\PP}{\mathbb{P}}
\newcommand{\lin}{\left\langle}
\newcommand{\rin}{\right\rangle}
\newcommand{\SSS}{\mathbb{S}}
\newcommand{\forces}{\Vdash}
\newcommand{\dom}{\text{dom}}
\newcommand{\osc}{\text{osc}}

\author{Tristan Bice}
\address{Federal University of Bahia, Salvador, Brazil}
\email{\texttt{Tristan.Bice@gmail.com}}
\thanks{Part of the research leading to the results of
this paper was  conducted with support of
the grant PVE Ci\^encia sem Fronteiras - CNPq (406239/2013-4) while the 
first named author was visiting the Univeristy of
S\~ao Paulo in December 2015. We would like to thank
Christina Brech for organizing the visit.} 
\thanks{The first named author was supported by an IMPA (Brazil) post-doctoral fellowship.}

\author{Piotr Koszmider}
\address{Institute of Mathematics, Polish Academy of Sciences,
ul. \'Sniadeckich 8,  00-656 Warszawa, Poland}
\email{\texttt{piotr.koszmider@impan.pl}}
\thanks{The second named author was  supported at the University of by S\~ao Paulo by  grant
PVE Ci\^encia sem Fronteiras - CNPq (406239/2013-4).} 

\title{A note on the Akemann-Doner and Farah-Wofsey constructions}

\begin{abstract} 
We  remove the assumption of
the continuum hypothesis from the Akemann-Doner construction of a non-separable $C^*$-algebra $A$ with only separable commutative $C^*$-subalgebras.  We also extend a result of Farah and Wofsey's, constructing $\aleph_1$ commuting projections in the Calkin algebra with no commutative lifting.  This
removes the assumption of the continuum hypothesis from a version of a result of Anderson.
Both results are based on Luzin's almost disjoint family construction.
\end{abstract}

\maketitle

\section*{Background}

Recall that an almost disjoint family is a family $\mathcal F$ of infinite subsets 
of $\N$ such that $A\cap B$ is finite for any distinct $A, B\in \mathcal F$. 
Uncountable almost disjoint families, known already to Hausdorff, Luzin and Sierpi\'nski
in the second decade of the 20th century carry sophisticated combinatorics.
Since the times of  Alexandroff and Urysohn's memoir \cite{urysohn} this combinatorics
has been employed in constructions of interesting mathematical structures.
Applications in topology include, for example, compact spaces of countable tightness which
are not Frechet, or  two Frechet compact spaces whose product is not Frechet (for
a recent survey see \cite{hrusak}). The use of almost disjoint families
in Banach space theory was initiated by Johnson and Lindenstrauss in \cite{johnson}
and followed by many authors (e.g., \cite{rosenthal}, \cite{mrowka}) 

In \cite{akemann-doner}, Akemann and Doner considered $C^*$-subalgebras
 of the $C^*$-algebra $\ell_\infty(\mathcal M_2)$ of bounded sequences
of $2\times2$ complex matrices obtained from almost disjoint families.  
Assuming the continuum hypothesis (abbreviated later as CH)  they
constructed an uncountable almost disjoint family which yielded
 the first example of a non-separable 
$C^*$-algebra with only separable commutative $C^*$-subalgebras.  

Later Popa (see \cite{popa} Corollary 6.7) proved that the reduced 
$C^*$-algebra of an uncountable free group is an example of such 
$C^*$-algebra whose existence does not require CH or any other 
set-theoretic assumption beyond the usual
axioms ZFC.  

We show in Theorem \ref{Akemann-DonerConstruction} 
that CH can in fact already be removed from the 
Akemann-Doner construction by considering a so-called Luzin family 
(see \cite{luzin}), putting it on a more equal footing with Popa's example.  
Indeed, while Popa's example is highly non-commutative (being simple, for example), 
the Akemann-Doner example is barely non-commutative (being $2$-subhomogeneous, for example).  
Thus we see that, even in ZFC, a $C^*$-algebra can be nearly commutative and 
yet only have small commutative $C^*$-subalgebras. Our version
of the Akemann-Doner construction has many other interesting features which
are the consequence of its relative elementarity, for example it is  
a subalgebra $A$ of the algebra $\mathcal B(H)$ of all bounded operators  on
a separable Hilbert space  $H$ which includes 
a separable ideal $J=A\cap \mathcal K(H)$, where
$\mathcal K(H)$  denotes the ideal of compact operators on $H$,
such that its quotient $A/J$ by $J$ is the commutative $C^*$-algebra $c_0(\omega_1)$
of all continuous functions on the discrete uncountable space $\omega_1$ vanishing at 
the infinity. In particular it is a scattered $C^*$-algebra in the sense of \cite{jensen}, while Popa's
example, as a simple $C^*$-algebra, has the opposite properties, for example has no minimal projections.

The second application of Luzin's family which we present in this note 
is related to a topic 
concerning the Calkin algebra $\mathcal B(H)/\mathcal K(H)$ of
bounded operators on the separable Hilbert space $H$
modulo the ideal of compact operators on $H$. This topic
can be traced back to the paper \cite{anderson-pathology} of
Anderson where assuming CH he constructed 
a maximal selfadjoint abelian  subalgebra (masa) of $\mathcal B(H)/\mathcal K( H)$
which cannot be lifted to a masa in $\mathcal B(H)$. 
In his proof Anderson constructed under CH an uncountable family $\mathcal P$ of commuting projections
in the Calkin algebra such that
no uncountable $\mathcal P_1\subseteq \mathcal P$ 
can be lifted to a family of commuting projections in
 $\mathcal B(H)$\footnote{If $\mathcal P$ is the almost central collection of projections from Theorem
4 of \cite{anderson-pathology}, then any of its uncountable
subsets is almost central as well. If an almost central
collection $\mathcal P$ could be lifted to a commuting collection
of projections $\mathcal P'$ in $\mathcal B(H)$, one could consider in $\mathcal B(H)$ a masa
$\mathcal A\supseteq \mathcal P'$. But then, by a theorem of 
Johnson and Parrott from \cite{johnson-parrott}, the algebra $\pi[\mathcal A]$ 
would be a masa in the Calkin algebra containing $\mathcal P$ and lifting
to  a masa in $\mathcal B$ which would contradict Proposition 3 of \cite{anderson-pathology}.}.

Echoing Luzin's construction, in Theorem 5.35 of \cite{farah-wofsey}, 
Farah and Wofsey constructed an $\aleph_1$-sized family of projections 
$\mathcal{P}$ in the Calkin algebra $\mathcal{C}(H)$ which can not be 
simultaneously diagonalized.  In fact, the proof shows that $\pi[A]\cap\mathcal{P}$ is countable,
 for all $C^*$-subalgebras $A$ of $\mathcal{B}(H)$ isomorphic to $l_\infty$, where $\pi$ is the 
canonical homomorphism from $\mathcal{B}(H)$
onto the Calkin algebra $\mathcal{C}(H)=\mathcal{B}(H)/\mathcal{K}(H)$.  
They conjectured that this could be extended to arbitrary commutative $C^*$-subalgebras 
$A$ of $\mathcal{B}(H)$.     Our main result
of Section 3, Theorem \ref{Farah-WofseyConstruction} proves this conjecture
and removes the assumption of CH from the above version of 
 Anderson's result. That is, without any additional
set-theoretic assumptions we construct in the Calkin algebra
an uncountable family $\mathcal P$ of commuting projections
 such that
no uncountable $\mathcal P_1\subseteq \mathcal P$ 
can be lifted to a family of commuting projections in
 $\mathcal B(H)$. We would like to thank both Ilijas Farah and Joerg Brendle for various discussions related to this 
part of our work. 

We should mention that Akemann and Weaver noted at the end of \cite{akemann-weaver} that, regardless of CH, there must be $2^{2^{\aleph_0}}$ masas in the Calkin algebra which do not lift to masas in $\mathcal{B}(H)$.  However, these masas may not be generated by projections, as in Anderson's construction.  They also have density $2^{\aleph_0}$, while the key point of our construction is that only $\aleph_1$ projections suffices.  Also Luzin families have  been used recently to construct subalgebras of $\ell_\infty(M_2)$ with other interesting properties \textendash\, see \cite{choi-farah-ozawa} and \cite{vignati}.

\section{Luzin Families}

\begin{definition} A Luzin family is an almost disjoint family $\mathcal{L}=\{D_\alpha: \alpha<\omega_1\}$ such that, for every $\alpha<\omega_1$ and every $k\in \N$, the following set is finite.
$$\{\beta<\alpha: D_\beta\cap D_\alpha\subseteq \{0, ..., k\}\}$$
\end{definition}

For a ZFC construction of a Luzin family, see \cite{vandouwen} Theorem 4.1. or \cite{hrusak} Theorem 3.1.

Whenever $A, B\subseteq \N$, then we write respectively $A\subseteq^* B$ or $A\cap B=^*\emptyset$ if $A\setminus B$ is finite or $A\cap B$ is finite.  For the convenience of the reader let us recall a fundamental property of a Luzin family:

\begin{proposition}\label{luzin-property}
Suppose that $\mathcal{L}$ is a Luzin family and $\mathcal{L}', \mathcal{L}''\subseteq \mathcal{L}$ are uncountable and disjoint.  Then there is no $A\subseteq \N$ such that, for all $D'\in \mathcal{L}'$ and $D''\in \mathcal{L}''$,
$$D'\subseteq^* A\ \hbox{and}\ D''\cap A=^*\emptyset.$$
\end{proposition}
\begin{proof} If there is such an $A\subseteq \N$ then $X'=\{\alpha<\omega:D_\alpha\setminus\{1, ..., k'\}\subseteq A\}$ is uncountable, for some $k'\in \N$.  Likewise $X''=\{\alpha<\omega:D_\alpha\setminus\{1, ..., k'\}\cap A=\emptyset\}$ is uncountable, for some $k''\in \N$.
Let $k=\max(k', k'')$ and take $\alpha\in X'$ such that $X''\cap\alpha$ is infinite. 
But $D_\beta\cap D_\alpha\subseteq \{1, ..., k\}$ for every $\beta\in X''\cap\alpha$
which contradicts the definition of a Luzin family.
\end{proof}

\section{The Akemann-Doner Construction}

First note the following elementary $C^*$-algebra result.

\begin{lemma}\label{commuting||p-q||<1}
If projections $p$ and $q$ commute and $||p-q||<1$ then $p=q$.
\end{lemma}

\begin{proof} Using the standard argument based on the Gelfand-Naimark theorem
we will be working with $p$ and $q$ as with projections on some Hilbert space.
Defining $p^\perp=1-p$, we see that
\begin{equation*}
||p-q||=||p-pq+pq-q||=||pq^\perp-p^\perp q||=\max(||pq^\perp||,||p^\perp q||),
\end{equation*}
as $(pq^\perp)(p^\perp q)^*=pq^\perp qp^\perp=p0p^\perp=0$ and $(pq^\perp)^*(p^\perp q)=q^\perp pp^\perp q=q^\perp 0q=0$, i.e. $pq^\perp$ and $p^\perp q$ have orthogonal range and cokernel,
so the norm of their sum is the maximum of the norms.  As $pq=qp$, $pq^\perp$ is a projection so $||pq^\perp||=0$ or $1$ which, as $||p-q||<1$, means $pq^\perp=0$, i.e. $p=pq$.  Likewise, $p^\perp q$ is a projection so $q=pq$.
\end{proof}

In fact, for any $p,q\in\mathcal{P}^1=$ the rank one projections in $\mathcal M_2$, we have
\[||p-q||=||pq^\perp||=||p^\perp q||.\]
Thus using the fact that for any $a$ in the algebra  $||ap||$ is the supremum of all $||av||$ taken over unit vectors
 $v$ in the range of $p$, for any unit vector $v\in\mathcal{H}_2$ with $pv=v$ we have
\begin{equation}\label{pq}
1=||v||=||qv||^2+||q^\perp v||^2=||pq||^2+||pq^\perp||^2=||p-q^\perp||^2+||p-q||^2,
\end{equation}
since $||pq||=||(pq)^*||=||qp||$ for any two projections as the involution is isometric. We now make the following assumptions.

\begin{definition}\label{definition-main} $ $
\begin{itemize}
\item $\mathbf{p}\in\mathcal{P}^1$ is fixed throughout.
\item $\mathcal{L}$ is a Luzin almost disjoint family on $\mathbb{N}$.
\item $(p_D)_{D\in\mathcal{L}}\subseteq\mathcal{P}^1$ are distinct with $||p_D-\mathbf{p}||<\frac{1}{4}$, for all $D\in\mathcal{L}$.
\item $\pi$ is the canonical homomorphism from $\ell_\infty(\mathcal M_2)$ to $\ell_\infty(\mathcal M_2)/c_0(\mathcal M_2)$.
\item For $D\subseteq\N$, $P_D$ denotes the central projection in $\ell_\infty(\mathcal M_2)$ defined by
\begin{equation*}
P_D(n) = \begin{cases}
                 1             & n \in D\\
                 0             & n \notin D.\\
           \end{cases}
\end{equation*}
\item Elements of $\mathcal M_2$ are identified with constant functions in $\ell_\infty(\mathcal M_2)$.
\item $L$ is the $C^*$-subalgebra of $\ell_\infty(\mathcal M_2)$ generated by $(p_DP_D)_{D\in\mathcal{L}}$ and $c_0(\mathcal M_2)$.
\end{itemize}
\end{definition}

\begin{theorem}\label{Akemann-DonerConstruction}
$L$ is non-separable but only has separable commutative $C^*$-subalgebras.
\end{theorem}

\begin{proof} As $\pi(p_DP_D)$ is an uncountable pairwise orthogonal collection of projections, $\pi[L]$ and hence $L$ is nonseparable.  This also means, for any $a\in L$,
\[\pi(a)=\sum\lambda^D_a\pi(p_DP_D)\]
for unique $(\lambda^D)_{D\in\mathcal{L}}\subseteq\mathbb{C}$ with $\lambda^D\rightarrow0$ (on the countable subset $\{D\in\mathcal{L}:\lambda_D\neq0\}$).

Now suppose that $A$ is a non-separable $C^*$-subalgebra of $L$ which is commutative. We will get a contradiction with the property of a Luzin family from Lemma \ref{luzin-property}. 
 For each $n\in \N$, let $q(n)\in\mathcal{P}^1$ be such that $a(n)q(n)=q(n)a(n)$, for each $a\in A$.  
It exists since any nonzero commutative subalgebra of $\mathcal M_2$ must contain rank one projections
by the functional calculus argument.
By \ref{pq}, replacing $q(n)$ with $q(n)^\perp$ if necessary, 
we can also assume $||q(n)-\mathbf{p}||^2\leq\frac{1}{2}$. Note that
$q$ is in  $\ell_\infty(\mathcal M_2)$ but not necessarily in $L$ nor $A$.

Say $a\in A$ and $\lambda^D_a\neq0$ for some $D\in \mathcal L$.  As $aq=qa$, we have
$\pi(aq)=\pi(qa)$ and hence $\pi(aqP_D)=\pi(qaP_D)$.  As $\pi(P_DP_E)=0$, for all $E\in\mathcal{L}\setminus\{D\}$, this means that $\pi(\lambda^D_ap_DqP_D)=\pi(q\lambda^D_ap_DP_D)$ and hence $\pi(p_DqP_D)=\pi(qp_DP_D)$, i.e. $\pi(p_DP_D)$ and $\pi(qP_D)$ commute.  But $||q-\mathbf{p}||^2\leq\frac{1}{2}$ and $||p_D-\mathbf{p}||\leq\frac{1}{4}$ and hence $||q-p_D||<1$.  Thus $||\pi(qP_D)-\pi(p_DP_D)||\leq||(q-p_D)P_D||<1$ and hence $\pi(qP_D)=\pi(p_DP_D)$, by Lemma \ref{commuting||p-q||<1}.  So $\lim_{n\in D}||q(n)-p_D(n)||=0$.

As $A$ is non-separable, we must have  uncountably many  distinct $D\in \mathcal L$ for which
there is $a\in A$ with  $\lambda^D_a\neq0$.  This already will give as the desired contradiction.
 Thus we have uncountable $\mathcal{L}'\subseteq\mathcal{L}$ with $\lim_{n\in D}||q(n)-p_D(n)||=0$, for all $D\in\mathcal{L}'$.  As $\mathcal{P}^1$ is a separable metric space, $(p_D)_{D\in\mathcal{L}'}$ must have at least two distinct condensation points $r$ and $s$, i.e. such that every neighbourhood of $r$ and $s$ contain uncountably many $(p_D)_{D\in\mathcal{L}'}$.  Let
\begin{align*}
\mathcal{E}&=\{D\in\mathcal{L}':||p_D-r||<\tfrac{1}{2}||r-s||\}\text{ and}\\
\mathcal{F}&=\{D\in\mathcal{L}':||p_D-s||<\tfrac{1}{2}||r-s||\}.
\end{align*}
By the triangle inequality, $\mathcal{E}$ and $\mathcal{F}$ are disjoint, as are
\begin{align*}
X&=\{n\in\mathbb{N}:||q(n)-r||<\tfrac{1}{2}||r-s||\}\text{ and}\\
Y&=\{n\in\mathbb{N}:||q(n)-s||<\tfrac{1}{2}||r-s||\}.
\end{align*}
As $\lim_{n\in D}||q(n)-p_D(n)||=0$, for all $D\in\mathcal{L}'$, we see that $E\subseteq^*X$, for all $E\in\mathcal{E}$ and $F\subseteq^*Y$, for all $F\in\mathcal{F}$.  This contradicts Proposition \ref{luzin-property}.
\end{proof}

\section{The Farah-Wofsey Construction}

The following construction, based on \cite{farah-wofsey} Theorem 5.35, yields a family of projections in $\mathcal{B}(H)$ where $H$ is a separable Hilbert space with a Luzin-like property with respect to $(K_n)_{n\in \N}\subseteq\mathcal{K}(H)$.

\begin{lemma}\label{fw-lemma}
For every $(K_n)_{n\in \N}\subseteq\mathcal{K}(H)$ and every $0<\epsilon<1/2$ there are infinite rank projections $(P_\alpha)_{\alpha\in\aleph_1}\subseteq\mathcal{B}(H)$ such that, for all distinct $\alpha,\beta\in\aleph_1$, $P_\alpha P_\beta\in\mathcal{K}(H)$ and, for all $\beta\in\aleph_1$, all $n\in\N$, and all but possibly $n$ many $\alpha\in\beta$,
\[||(P_\alpha+K_n)(P_\beta+K_n)-(P_\beta+K_n)(P_\alpha+K_n)||\geq\epsilon.\]
\end{lemma}

\paragraph{Proof:}  We construct $(P_\alpha)_{\alpha\in\aleph_1}$ by recursion, as follows.  Let $P_0$ be any infinite rank projection with $P_0^\perp$ infinite rank too.  
Assume, for some $\gamma\in\aleph_1\backslash\{0\}$, $(P_\alpha)_{\alpha\in\gamma}$
 has already been constructed such that $\bigvee_{\alpha\in F}\pi(P_\alpha)\neq1$, for all
 finite $F\subseteq\gamma$, and, for all distinct $\alpha,\beta\in\gamma$, 
$P_\alpha P_\beta\in\mathcal{K}(H)$.  In particular, $\pi(P_\alpha)$ and $\pi(P_\beta)$
 commute for all $\alpha,\beta\in\gamma$ and hence, by \cite{farah-wofsey} Lemma 5.34, there exists an orthonormal basis $(e_n)_{n\in \N}$ of $H$ and $(A_n)_{n\in \N}\subseteq\wp(\N)$ 
such that $\pi(p_{A_n})=\pi(P_{\alpha_n})$ for all $n\in\min(\gamma,\omega)$, where $n\mapsto\alpha_n$ is any 
fixed one-to-one mapping of $\min(\gamma,\omega)$ onto $\gamma$ and $p_X$ denotes the projection onto $\overline{\mathrm{span}}(e_n)_{n\in X}$  for $X\subseteq \N$.

Take $\delta>0$ with $\epsilon\leq\frac{1}{2}-(\frac{1}{\sqrt{2}}+2)\delta$, and recursively define an increasing sequence $(k_n)_{n\in\min(\gamma,\omega)}\subseteq\N$ as follows.  Let $k_0$ be large enough that
\begin{eqnarray*}
||(P_{\alpha_0}+K_0-p_{A_0})p^\perp_{k_0}|| &<& \delta\qquad\textrm{and}\\
||((P_{\alpha_0}+K_0)K_0-K_0(P_{\alpha_0}+K_0))p^\perp_{k_0}|| &<& \delta.
\end{eqnarray*}
Once $(k_n)_{n\leq m}$ has been defined, let $k_{m+1}>k_m$ be large enough that there exists distinct $i(m),j(m)\in k_{m+1}\backslash(k_m\cup\bigcup_{n\in m}A_n)$ such that $j(m)\notin A_m$ and $i(m)\in A_m$ and, for all $l\leq m+1$,
\begin{eqnarray*}
||(P_{\alpha_{m+1}}+K_l-p_{A_{m+1}})p^\perp_{k_{m+1}}|| &<& \delta,\\
||((P_{\alpha_{m+1}}+K_l)K_l-K_l(P_{\alpha_{m+1}}+K_l))p^\perp_{k_{m+1}}|| &<& \delta.
\end{eqnarray*}
Note that, for sufficiently large $k_{m+1}$, there will always exist such a $j(m)$ because
 $\N\backslash\bigcup_{n\leq m}A_n$ is infinite which, in turn, 
follows from the fact that $\bigvee_{n\leq m}\pi(p_{A_n})=
\bigvee_{n\leq m}\pi(P_{\alpha_n})\neq1$.  If $\gamma$ is finite then simply let $i(m)$ and $j(m)$ all
 be distinct elements of $\N\backslash\bigcup_{n\in\gamma}A_n$, for $m\geq\gamma$.

Now let $P_\gamma$ be the projection onto $\overline{\mathrm{span}}\{e_{i(n)}+e_{j(n)}:n\in\N\}$.  
Note that, for all $m\in\min(\gamma,\omega)$ and $n>m$, we have $i(n),j(n)\notin A_m$ so \[p_{A_m}P_\gamma[H]\subseteq\mathrm{span}(\{e_{i(n)}:n\leq m\}\cup\{e_{j(n)}:n\leq m\})\] so $P_{\alpha_m}P_\gamma\in\mathcal{K}(H)$.  Also  $P_\gamma e_{i(m)}=\frac{1}{2}(e_{i(m)}+e_{j(m)})$.  For any $m\in\min(\gamma,\omega)$, $i(m)\in A_m$ and so $p_{A_m}(e_{i(m)}+e_{j(m)})=e_{i(m)}=p_{A_m}e_{i(m)}$ and hence, for all $l\leq m$,

$$||(P_{\alpha_m}+K_l)P_\gamma e_{i(m)}-\tfrac{1}{2}e_{i(m)}||=$$
$$= ||(P_{\alpha_m}+K_l)(\tfrac{1}{2}(e_{i(m)}+e_{j(m)}))-p_{A_m}(\tfrac{1}{2}(e_{i(m)}+e_{j(m)}))||=$$
$$=  ||(P_{\alpha_m}+K_l-p_{A_m})(\tfrac{1}{2}(e_{i(m)}+e_{j(m)}))||
\leq\delta||e_{i(m)}+e_{j(m)}||/2
\leq \delta/\sqrt{2},$$
and
$$||P_\gamma(P_{\alpha_m}+K_l)e_{i(m)}-\tfrac{1}{2}(e_{i(m)}+e_{j(m)})||
= ||P_\gamma(P_{\alpha_m}+K_l)e_{i(m)}-P_\gamma p_{A_m}e_{i(m)}||\leq$$
$$\leq ||(P_{\alpha_m}+K_l-p_{A_m})e_{i(m)}||
\leq\delta.$$
so
$$||(P_{\alpha_m}+K_l)(P_\gamma+K_l)e_{i(m)}-(P_\gamma+K_l)(P_{\alpha_m}+K_l)e_{i(m)}||\geq$$
$$\geq ||\tfrac{1}{2}e_{j(m)}||-||(P_{\alpha_m}+K_l)P_\gamma e_{i(m)}-\tfrac{1}{2}e_{i(m)}||
 -||P_\gamma(P_{\alpha_m}+K_l)e_{i(m)}-\tfrac{1}{2}(e_{i(m)}+e_{j(m)})||+$$
$$-||((P_{\alpha_m}+K_l)K_l-K_l(P_{\alpha_m}+K_l))e_{i(m)}||
\geq 1/2-\delta/\sqrt{2}-\delta-\delta
\geq \epsilon.$$
Thus $(P_{\alpha})_{\alpha\leq\gamma}$ satisfies the required conditions.

Also note that if $\gamma$ is finite then the projection $Q$ onto $\overline{\mathrm{span}}(e_{i(n)}-e_{j(n)})_{n\geq\gamma}$ is orthogonal to $P_\gamma$ and $p_{A_n}$, for each $n\in\gamma$.  This means $P_\gamma\leq Q^\perp$ and $p_{A_n}\leq Q^\perp$, for each $n<m$, and hence $\pi(P_\alpha)\leq\pi(Q^\perp)$ for all $\alpha\leq\gamma$ which means $\bigvee_{\alpha\leq\gamma}\pi(P_\alpha)\leq\pi(Q^\perp)<1$.  Thus the recursion can be continued for finite $\gamma$.  On the other hand, if $\gamma$ is not finite and $(P_\alpha)_{\alpha<\gamma}$ has already been constructed then, for any $F\in[\gamma]^{<\aleph_0}$, we can find $\beta\in\gamma\backslash F$ and then $\bigvee_{\alpha\in F}\pi(P_\alpha)\leq\pi(P^\perp_\beta)<1$, so the recursion can also be continued for infinite $\gamma$. $\Box$\\

\begin{theorem}\label{Farah-WofseyConstruction}
There are $\aleph_1$ orthogonal projections in $\mathcal{B}(H)/\mathcal{K}(H)$ containing no uncountable subset that simultaneously lifts to commuting projections in $\mathcal{B}(H)$.
\end{theorem}

\paragraph{Proof:}  Take dense $(K_n)_{n\in \N}\subseteq\mathcal{K}(H)$ and $\epsilon\in(0,1/2)$ and let $(P_\alpha)_{\alpha\in\aleph_1}$ be obtained from Lemma \ref{fw-lemma}.  Assume that we have some uncountable $A\subseteq\aleph_1$ and $(K'_\alpha)_{\alpha\in A}\subseteq\mathcal{K}(H)$ such that $(P_\alpha+K'_\alpha)_{\alpha\in A}$ commute.  By replacing $A$ with an uncountable subset of $A$ if necessary, we may assume that there exists some $M\in\mathbb{R}$ such that $||K'_\alpha||\leq M$ for all $\alpha\in A$.  Take $\delta>0$ with $2\delta(1+M)+2(1+M+\delta)\delta<\epsilon$, and pick $n_\alpha\in\N$ such that $||K_{n_\alpha}-K'_\alpha||\leq\delta$, for all $\alpha\in A$.  Again replacing $A$ with an uncountable subset of $A$ if necessary, we may assume that there exists some $n\in\N$ such that $K_{n_\alpha}=K_n$ for all $\alpha\in A$.  Then, for any $\beta\in A$ and $\alpha\in A\cap\beta$, as we have $(P_\alpha+K'_\alpha)(P_\beta+K'_\beta)-(P_\beta+K'_\beta)(P_\alpha+K'_\alpha)=0$, we obtain
\begin{eqnarray*}
&& ||(P_\alpha+K_n)(P_\beta+K_n)-(P_\beta+K_n)(P_\alpha+K_n)||\\
&\leq& ||K'_\alpha-K_n||||P_\beta+K_n||+||P_\alpha+K'_\alpha||||K'_\beta-K_n||\\
&&+\ ||K'_\beta-K_n||||P_\alpha+K_n||+||P_\beta+K'_\beta||||K'_\alpha-K_n||\\
&\leq& \delta(1+M)+(1+M+\delta)\delta+\delta(1+M)+(1+M+\delta)\delta\\
&<& \epsilon.
\end{eqnarray*}
But for any $\beta$ such that $A\cap\beta$ contains more than $n$ elements, this contradicts the defining property of the $(P_\alpha)_{\alpha\in\aleph_1}$. $\Box$\\

\bibliographystyle{amsplain}

\begin{thebibliography}{20}

\bibitem{urysohn} P. Alexandroff and P. Urysohn. \emph{Memoire sur les espaces topologiques compacts
 dedie a
Monsieur D. Egoroff}., 1929.

\bibitem{akemann-doner} C. Akemann, J.  Doner, 
\emph{A nonseparable $C^*$-algebra with only separable abelian $C^*$-subalgebras}.
Bull. London Math. Soc. 11 (1979), no. 3, 279--284. 

\bibitem{akemann-weaver} C.  Akemann, N.  Weaver, \emph{$\mathcal B(H)$ has a pure state that is not multiplicative on any masa}. Proc. Natl. Acad. Sci. USA 105 (2008), no. 14, 5313--5314. 

\bibitem{anderson-pathology} J. Anderson, \emph{Pathology in the Calkin algebra}. 
J. Operator Theory 2 (1979), no. 2, 159--167. 



\bibitem{choi-farah-ozawa}  Y. Choi, I. Farah, N. Ozawa, \emph{A nonseparable amenable operator algebra which is not isomorphic to a $C^*$-algebra}.  Forum of Mathematics, Sigma. Vol. 2. Cambridge University Press, 2014 {\tt http://dx.doi.org/10.1017/fms.2013.6}.

\bibitem{farah-wofsey} I. Farah,  E. Wofsey,
\emph{Set theory and operator algebras}, in Appalachian Set Theory 2006-2012, eds.
J. Cummings, E. Schimmerling, LMS Lecture Notes Series 406.  {\tt http://www.math.cmu.edu/\textasciitilde eschimme/Appalachian/FarahWofseyNotes.pdf}

\bibitem{jensen}  H. Jensen, \emph{Scattered $C^*$-algebras}. 
Math. Scand. 41 (1977), no. 2, 308--314.



\bibitem{johnson-parrott} B. Johnson,  S. Parrott, 
\emph{Operators commuting with a von Neumann algebra modulo the set of compact operators}.
J. Functional Analysis 11 (1972), 39--61.

\bibitem{mrowka} P. Koszmider, 
\emph{On decompositions of Banach spaces of continuous functions on Mr\'owka's spaces}. 
Proc. Amer. Math. Soc. 133 (2005), no. 7, 2137--2146. 

\bibitem{hrusak}  M. Hrusak, \emph{Almost disjoint families and topology}. 
Recent progress in general topology. III, 601--638, Atlantis Press, Paris, 2014.

\bibitem{johnson} W. Johnson, J. Lindenstrauss; \emph{Some remarks 
on weakly compactly generated Banach spaces},
Israel J. Math. 17, 1974,  219--230, and Israel J. Math. 32 (1979), no. 4, 382--383.

\bibitem{luzin} N. Luzin, 
\emph{On subsets of the series of natural numbers}. (Russian)
Izvestiya Akad. Nauk SSSR. Ser. Mat. 11, (1947), 403--410. 



\bibitem{popa}  S. Popa, \emph{Orthogonal pairs of $*$-subalgebras in
 finite von Neumann algebras}. J. Operator Theory 9 (1983), no. 2, 253--268. 

\bibitem{rosenthal} H. Rosenthal, \emph{On relatively disjoint families of 
measures with some applications to Banach
space theory}, Studia Math. 37, (1970), pp. 13--36.

\bibitem{vandouwen} E. van Douwen, \emph{The integers and topology}, in: K. Kunen, J.E. Vaughan (Eds.), Handbook of Set-Theoretic Topology, North-Holland, Amsterdam, 1984, pp. 111--167.

\bibitem{vignati} Vignati, Alessandro. An algebra whose subalgebras are characterized by density. The Journal of Symbolic Logic, 80 (2015), 1066--1074. {\tt http://dx.doi.org/10.1017/jsl.2014.86}.
\end{thebibliography}

\end{document}